\documentclass{ws-ijbc}
\usepackage{ws-rotating}     
\usepackage{graphicx}
\usepackage{epstopdf}

\usepackage{amsmath}
\begin{document}

\catchline{}{}{}{}{} 

\markboth{Lijuan Sheng}{Limit cycles of a class of piecewise smooth Li\'enard systems}

\title{\boldmath \uppercase{Limit cycles of a class of piecewise smooth Li\'enard systems}}

\author{LIJUAN SHENG}

\address{Department of Mathematics, Shanghai Normal University\\ Shanghai 200234, P. R. China\\
shenglj@shnu.edu.cn}

\maketitle

\begin{history}
\received{(to be inserted by publisher)}
\end{history}

\begin{abstract}
In this paper, we study the problem of limit cycle bifurcation in two piecewise polynomial systems of Li\'enard type with multiple parameters.
Based on the developed Melnikov function theory, we obtain the maximum number of limit cycles of these two systems.
\end{abstract}


\section{Introduction}
The study of the polynomial Li\'enard differential equation
$\ddot{x}+\widetilde{f}(x)\dot{x}+\widetilde{g}(x)=0$ has a very long history, where $\widetilde{f}$ and $\widetilde{g}$ are polynomials with degree $m,~n$ respectively.
Since it was first introduced in \cite{Lie}, many researchers have concentrated on its maximum number of limit cycles and their location \cite{HR,LL,X2,junmin}. As it is known, the equation is equivalent to a  planar system of the form
\begin{equation}\label{led1}
\dot{x}=y,~\dot{y}=-\widetilde{f}(x)y-\widetilde{g}(x).
\end{equation}

In recent years, an increasing amount of people are interested in finding the maximum number of limit cycles of piecewise smooth dynamical systems on the plane. One of the main methods is the Melnikov function method developed in \cite{LH,MS}. In addition, as an attempt to make further investigation of piecewise near-Hamiltonian system, the authors of \cite{HX} and \cite{X} introduced a new parameter $\lambda$ in the system as well as its Melnikov function. It turns out that, in some cases, by using the Melnikov function depending on the new parameter, one might find more limit cycles. In this paper, based on the main results of \cite{X,LH}, we will study the piecewise situation of \eqref{led1}. More precisely, we suppose that system \eqref{led1} is defined on two half-planes $G^+$ and $G^-$ separated by a straight line $l$, and take $\tilde{f},~\tilde{g}$ in \eqref{led1} as 
\begin{equation*}
\widetilde{f}(x)=\epsilon f(x),
~~\widetilde{g}(x)=\left\{\begin{aligned}
&x + \epsilon  g(x),&&(x,y)\in G^+,\\
&x - \epsilon  g(x),&&(x,y)\in G^-,
\end{aligned}
\right.
\end{equation*}
where $f$, $g$ are polynomials of degree $m$, $n$ respectively, and $\epsilon>0$ is a small parameter. Then system \eqref{led1} becomes
\begin{equation}\label{main}
\left(\begin{aligned}&\dot{x}\\&\dot{y}\end{aligned}\right)
=\left\{ \begin{aligned}
&\left( \begin{aligned}
&y\\
&-x-\epsilon \Big(yf(x)+g(x)\Big)
\end{aligned}
\right),        &&(x,y)\in G^+,\\
&\left( \begin{aligned}
&y\\
&-x-\epsilon \Big(yf(x)-g(x)\Big)
\end{aligned}
\right),     &&(x,y)\in G^-.
\end{aligned}
\right.
\end{equation}

The following theorems are the main results of this paper:
\begin{theorem}\label{thi}
Let $G^{\pm}=\{(x,y)|y\in \mathbb{R}^{\pm}\}$. Then system \eqref{main} can have $[\frac{m}{2}]+2[\frac{n}{2}]+1$ limit cycles bifurcating from the periodic orbits of the system $\dot{x}=y,~\dot{y}=-x$.\end{theorem}
\begin{theorem}\label{thii}
Let $G^{\pm}=\{(x,y)|x\in \mathbb{R}^{\pm}\}$. If $n\geq1$ $($ $n=0$, resp.$)$, then system \eqref{main} can have $2[\frac{m}{2}]+[\frac{n+1}{2}]$  $($ $2[\frac{m}{2}]+1$, resp.$)$ limit cycles bifurcating from the periodic orbits of the system $\dot{x}=y,~\dot{y}=-x$.
\end{theorem}
The above two theorems will be proved in Sections 3 and 4.

\section{Melnikov functions with multiple parameters}

 Consider the following near-Hamiltonian system with multiple parameters
\begin{equation}\label{eqx}
\left(\begin{aligned}&\dot{x}\\&\dot{y}\end{aligned}\right)
=\left\{ \begin{aligned}
&\left( \begin{aligned}
&H_y^+(x,y,\lambda)+\epsilon p^+(x,y,\lambda)\\
&-H_x^+(x,y,\lambda)+\epsilon q^+(x,y,\lambda)
\end{aligned}
\right),        &&x\geq0,\\
&\left( \begin{aligned}
&H_y^-(x,y,\lambda)+\epsilon p^-(x,y,\lambda)\\
&-H_x^-(x,y,\lambda)+\epsilon q^-(x,y,\lambda)
\end{aligned}
\right),     &&x<0,
\end{aligned}
\right.
\end{equation}
where $0< \epsilon \ll \lambda \ll 1$, $H^{\pm}, p^{\pm}, q^{\pm}$ are $C^\infty$ functions of the following expressions
\begin{align*}
H^{\pm}(x,y,\lambda)&=H_0^{\pm}(x,y)+\lambda H_1^{\pm}(x,y)+O(\lambda^2),\\
p^{\pm}(x,y,\lambda)&=p_0^{\pm}(x,y)+\lambda p_1^{\pm}(x,y) +O(\lambda^2),\\
q^{\pm}(x,y,\lambda)&=q_0^{\pm}(x,y)+\lambda q_1^{\pm}(x,y)+O(\lambda^2).
\end{align*}
For the above system, we make the following assumptions:
\begin{alphlist}[(d)]
\item For small $\lambda$, there exists an open interval $I_\lambda$ such that \eqref{eqx}$|_{\epsilon=0}$ has a family of periodic orbits with clockwise orientation given by
$$L_\lambda(h): ~~~~H(x,y,\lambda)=h,~~h\in I_\lambda.$$
\item Each periodic orbit $L_\lambda(h)$ defined in (a) intersects the $y$-axis with two different points in turn, denoted by $A_\lambda(h)=(0,a(h,\lambda))$ and $B_\lambda(h)=(0,b(h,\lambda))$ with $a(h,\lambda)>b(h,\lambda)$. Denote that $A(h)=A_\lambda(h)|_{\lambda=0}$ and $B(h)=B_\lambda(h)|_{\lambda=0}$.
\end{alphlist}

 Consider the orbit of system \eqref{eqx} starting from $A_\lambda$. Let $B_\lambda^\epsilon$ denote its intersection point with the positive $y$-axis, see Fig. \ref{fig1}. Then \cite{LH} introduced a function $F$ by the following
    \begin{equation*}\label{eqF}
    H^+(B_\lambda^\epsilon)-H^+(A_\lambda)=\epsilon F(h,\lambda,\epsilon).
    \end{equation*}

\begin{figure}[h]
\centering
\includegraphics[width=5cm]{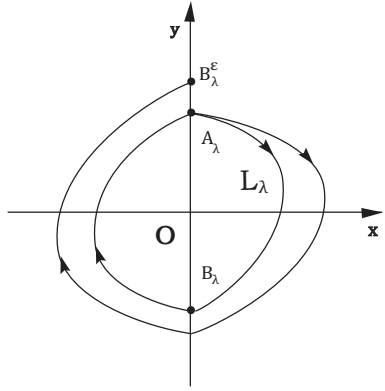}
\caption {The Poincar\'e map of system \eqref{eqx}.}\label{fig1}
\end{figure}

The function $F(h,\lambda,\epsilon)$ is called a bifurcation function of system \eqref{eqx}. Then, define the first order Melnikov function of \eqref{eqx} by $M(h,\lambda)=F(h,\lambda,0)$. 
Under the assumptions (a) and (b), \citet{LH} and \citet{MS} proved the following results respectively.
\begin{lemma}[Liu \& Han, 2010]\label{lemlh}
For $|\epsilon|$ small and $h\in I_\lambda$, $F(h,\lambda,\epsilon)\in C^\infty$. Moreover, \eqref{eqx} has a periodic solution $($a limit cycle, resp.$)$ near $L_\lambda(h_0)$ for $h_0\in I_\lambda$ if and only if $F$ has a zero $($an isolated zero, resp.$)$ in $h$ near $h_0$.
\end{lemma}

\begin{lemma}[Han \& Sheng, 2015]\label{lemhan} 
\begin{romanlist}
\item If $M(h,\lambda)$ has $k$ zeros in $h$ on the interval $I_{\lambda}$ with each having an odd multiplicity, then \eqref{eqx} has at least $k$ limit cycles bifurcating from the period annulus for $\varepsilon$ small;
\item If the function $M(h,\lambda)$ has at most $k$ zeros in h on the interval $I_{\lambda}$, taking into multiplicities account, then there exist at most $k$ limit cycles of \eqref{eqx} bifurcating from the period annulus.
\end{romanlist}
\end{lemma}

\citet{LH} aslo obtained the formula of $M(h,\lambda)$, which is simplified in \cite{LiangH} as follows
\begin{equation*}
M(h,\lambda)=\int\limits_{\widehat{A_\lambda B_\lambda}}q^+dx-p^+dy+\frac{H_y^+(A_\lambda,\lambda)}{H_y^-(A_\lambda,\lambda)}\int\limits_{\widehat{B_\lambda A_\lambda}}q^-dx-p^-dy.
\end{equation*}
Further, \citet{X} proved the following
\begin{lemma}[Xiong, 2015]\label{lemx} For $0<\lambda \ll 1$, we have
\begin{equation*}
M(h,\lambda)=M_0(h)+\lambda M_1(h)+O(\lambda^2),
\end{equation*}
where
\begin{align*}\label{eqm}
M_0(h)=&\int_{\widehat{AB}}q^+dx-p^+dy\big|_{\lambda=0}+\frac{H_y^+(A,0)}{H_y^-(A,0)}\int_{\widehat{BA}}q^-dx-p^-dy\big|_{\lambda=0}, \nonumber\\
M_1(h)=&-\int_{\widehat{AB}}H_1^+(p^+_{0x}+q^+_{0y})dt+\int_{\widehat{AB}}q_1^+dx-p_1^+dy+\mathcal{L}(p_0^+)\big|_{\lambda=0} \nonumber \\
&+\frac{H_{0y}^+(A)}{H_{0y}^-(A)}\bigg[-\int_{\widehat{BA}}H_1^-(p^-_{0x}+q^-_{0y})dt+\int_{\widehat{BA}}q_1^-dx-p_1^-dy-\mathcal{L}(p_0^-)\big|_{\lambda=0}\bigg]\nonumber\\
&+\Bigg[\frac{H_{y}^+(A_\lambda,\lambda)}{H_{y}^-(A_\lambda,\lambda)}\Bigg]_\lambda\Bigg|_{\lambda=0}\int_{\widehat{AB}}q_0^-dx-p_0^-dy,
\end{align*}
where for a $C^{\infty}$ function $r(x,y)$
\begin{equation}\label{eqL}
\mathcal{L}(r)=r(0,a(h,\lambda))\frac{\partial a}{\partial \lambda}-r(0,b(h,\lambda))\frac{\partial b}{\partial \lambda}.
\end{equation}
\end{lemma}

\section{Proof of Theorem \ref{thi}}
In this section we present a proof to Theorem \ref{thi}. Consider the following piecewise system with multiple parameters
\begin{equation}\label{0}
\left\{\begin{aligned}
&\dot{x}=y, \\
&\dot{y}=-x-\lambda \mbox{sgn}(y) g(x)-\epsilon \bigg[ y \Big(f_0(x)+\lambda f_1(x) \Big)+\mbox{sgn}(y)\Big(g_0(x)+\lambda g_1(x)\Big) \bigg],
\end{aligned}\right.
\end{equation}
where
\begin{align}\label{1-fg}
&f_i(x)=\sum_{j=0}^{m} a_j^{(i)} x^j,~~~~~g_i(x)=\sum_{j=0}^{n} b_j^{(i)} x^j,~~~~~~~i=0,1,\nonumber\\
&g(x)=\sum_{j=0}^{n} c_j x^j.
\end{align}

In order to apply Lemma \ref{lemx}, we make a change of variables of the form $x=\tilde{y},~y=\tilde{x}$ and time rescaling $\tau=-t$. For the sake of convenience, we still use $x$ and $y$ to replace $\tilde{x},~\tilde{y}$ respectively. Then system \eqref{0} can be transformed into
\begin{equation}\label{1}
\left\{\begin{aligned}
&\dot{x}=y+\lambda \mbox{sgn}(x) g(y)+\epsilon\bigg[x\Big(f_0(y)+\lambda f_1(y)\Big)+\mbox{sgn}(x) \Big(g_0(y)+\lambda g_1(y) \Big)\bigg], \\
&\dot{y}=-x.
\end{aligned}\right.
\end{equation}

Then by Lemma \ref{lemx}, the first order Melnikov function of \eqref{1} can be written as
\begin{equation*}
M(h,\lambda)=M_0(h)+\lambda M_1(h)+O(\lambda^2),
\end{equation*}
where
\begin{align}\label{1-m}
&M_0(h)=I_0(h),\nonumber\\
&M_1(h)=I_1(h)+I_2(h)+I_3(h)+I_4(h),
\end{align}
where
\begin{align}\label{1-I}
&I_i(h)=\int_{\widehat{AB}}-\big(xf_i(y)+g_i(y)\big)dy+\int_{\widehat{BA}}-\big(xf_i(y)-g_i(y)\big)dy,~i=0,1, \nonumber\\
&I_2(h)=-\int_{\widehat{AB}}H^+_1f_0(y)d\tau+\int_{\widehat{BA}}H^+_1f_0(y)d\tau,~\mbox{where}~ H^+_1=\int_{0}^{y}g(y)dy,\nonumber\\
&I_3(h)=\mathcal{L} \big(xf_0(y)+g_0(y)\big) |_{\lambda=0}-\mathcal{L}\big(xf_0(y)-g_0(y)\big)|_{\lambda=0}, \nonumber\\
&I_4(h)=\Bigg[ \frac{a(h,\lambda)+\lambda g\big(a(h,\lambda)\big)}{a(h,\lambda)-\lambda g\big(a(h,\lambda)\big)}\Bigg]_\lambda \Bigg|_{\lambda=0} \int_{\widehat{AB}}-\big(xf_0(y)-g_0(y)\big)dy,\nonumber\\
&A=(0,\sqrt{2h}),~~ B=(0,-\sqrt{2h}).
\end{align}
One can  easily find that \eqref{1}$_{\epsilon=0}$ has a first integral
$$H(x,y,\lambda)=\frac{1}{2}(x^2+y^2)+\lambda\mbox{sgn}(x)\int_0^yg(y)dy$$
and \eqref{1}$_{\epsilon=0,\lambda=0}$ has a family of periodic orbits $L_0(h):~x^2+y^2=2h$, see Fig \ref{fig2}.
\begin{figure}[h!]
\centering
\includegraphics[width=5cm]{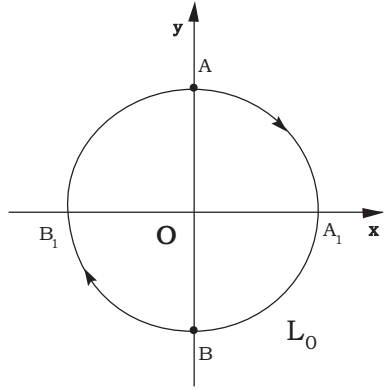}
\caption {The phase portrait of $L_0(h).$}\label{fig2}
\end{figure}
Then, we have a lemma below:

\begin{lemma}\label{lem1-M0}
The function $M_0(h)$ has at most $[\frac{m}{2}]+[\frac{n}{2}]+1$ zeros if it is not zero identically, and the number $[\frac{m}{2}]+[\frac{n}{2}]+1$ can be reached for some $f_0$ and $g_0$. Further, if the polynomials $f_0$ and $g_0$ are odd, then $M_0(h)\equiv0.$
\end{lemma}
\begin{proof}
By \eqref{1-fg}, \eqref{1-I} and Green formula, for $i=0,1,$ we have
\begin{align}\label{1-Ii}
I_i(h)&=\int_{\widehat{AB}}-\big(xf_i(y)+g_i(y)\big)dy+\int_{\widehat{BA}}-\big(xf_i(y)-g_i(y)\big)dy, \nonumber\\
&=-\oint_{L_0}xf_i(y)dy-2\int_{\widehat{AB}}g_i(y)dy \nonumber\\
&=-\sum_{j=0}^{m}a^{(i)}_j \oint_{L_0}y^j xdy+2 \sum_{j=0}^{n} b^{(i)}_j \int_{-\sqrt{2h}}^{\sqrt{2h}}y^jdy\nonumber\\
&=\sum_{j=0}^{m}a^{(i)}_j J_{1j}+\sum_{i=0}^{n}b^{(i)}_j J_{2j},
\end{align}
where
\begin{align*}
&J_{1j}=-\oint_{L_0}y^j xdy,&&j=0,\cdots,m,\\
&J_{2j}=2 \int_{-\sqrt{2h}}^{\sqrt{2h}}y^jdy,&&j=0,\cdots,n.
\end{align*}
By Green formula and using polar coordinate transformation $x=r\cos \theta,~y=r \sin \theta$, for $j=0,\ldots, m$, $J_{1j}$ has the form as follows
\begin{equation}\label{1-J1j}
J_{1j}=\iint\limits_{Int L_0}y^j dxdy=\int_{0}^{2 \pi}\int_{0}^{\sqrt{2h}} r^j (\sin\theta)^j r dr d\theta= \frac{(\sqrt{2h})^{j+2}\int_{0}^{2 \pi} (\sin\theta)^j d\theta}{j+2}.
\end{equation}
Note that
\begin{equation}\label{1-sin}
\int_{0}^{2 \pi} (\sin\theta)^j d\theta=\left\{
\begin{aligned}
&0,&& j \mbox{~ odd,}\\
&2 \pi \prod\limits_{l=1}^{j/2}\frac{2l-1}{2l},
&& j \mbox{~ even,}
\end{aligned}
\right.
\end{equation}
and that
\begin{equation}\label{1-J2j}
J_{2j}=\frac{2y^{j+1}}{j+1}\bigg|_{-\sqrt{2h}}^{\sqrt{2h}}=\left\{
\begin{aligned}
&0,&&j \mbox{~odd},\\
&\frac{4(\sqrt{2h})^{j+1}}{j+1},&&j \mbox{~even}.
\end{aligned}
\right.
\end{equation}

Taking \eqref{1-sin} into \eqref{1-J1j}, then \eqref{1-J1j}, \eqref{1-J2j} into \eqref{1-Ii}, we can obtain
\begin{equation}\label{1-final-Ii}
I_i(h)=\sum\limits_{j=0}^{[\frac{m}{2}]}\tilde{a}^{(i)}_{j}h ^{j+1}+
\sum\limits_{j=0}^{[\frac{n}{2}]}\tilde{b}^{(i)}_{j}h^{j+\frac{1}{2}},~~i=0,1,
\end{equation}
where
\begin{align}\label{1-tildeab}
&\tilde{a}^{(i)}_{j}=\bigg(\frac{2 \pi}{j+1} \prod\limits_{l=1}^{j}\frac{2l-1}{l}\bigg)a^{(i)}_{2j},~\tilde{a}^{(i)}_0=2\pi a^{(i)}_0,&&j=1,2,\cdots,[\frac{m}{2}], \nonumber\\
&\tilde{b}^{(i)}_{j}=\bigg(\frac{2^{j+\frac{5}{2}}}{2j+1}\bigg)b^{(i)}_{2j},&&j=0,1,\cdots,[\frac{n}{2}].
\end{align}

Specially, when $i=0,$ we have $M_0(h)=I_0(h)$.
Since $\tilde{a}^{(0)}_{j},~\tilde{b}^{(0)}_{j}$ are independent, by taking them as free parameters, we can conclude that $M_0(h)$ has at most $[\frac{m}{2}]+[\frac{n}{2}]+1$ isolated positive zeros for $h>0$ and this number can be reached. Furthermore, the second conclusion of this lemma is obviously obtained from \eqref{1-final-Ii} and \eqref{1-tildeab}.
\end{proof}

\begin{remark}
Resently, \citet{MM} consider the system
\begin{equation}\label{eqMM}
\left\{\begin{aligned}
&\dot{x}=y,\\
&\dot{y}=-x-\epsilon\Big(f(x)y+\mbox{sgn}(y)(k_1x+k_2)\Big),
\end{aligned}\right.
\end{equation}
where $f(x)=\sum_{i=0}^{m}a_ix^i$.

They proved that the above system has at most $[m/2]+1$ limit cycles by using the developed first-order average theory in \cite{bj2004}.
If we take $g(x)=k_1x+k_2$ and $n=1$, then by Lemmas \ref{lemhan}, \ref{lem1-M0} and Eq. \eqref{1-tildeab}, one can  get their result immediately, while the Melnikov function of system \eqref{eqMM} is
\begin{equation*}
M(h)=B_0h^\frac{1}{2}+B_1h+B_2h^2+ \cdots + B_{[\frac{m}{2}]+1}h^{[\frac{m}{2}]+1},
\end{equation*}
where
\begin{equation*}
B_0=4\sqrt{2}k_2,~B_1=2\pi a_0,~B_i=\Big(\frac{2 \pi}{i} \prod\limits_{l=1}^{i-1}\frac{2l-1}{l}\Big)a_{2i-2}, ~~i=2,\cdots, [\frac{m}{2}]+1.
\end{equation*}
\end{remark}

\begin{lemma}\label{lem1-M1}
Suppose that $f_0,~g_0$ are odd functions with degree at most $m,~n$ respectively. Then the function $M_1(h)$ has at most $[\frac{m}{2}]+2[\frac{n}{2}]+1$ zeros and the number can be reached.
\end{lemma}
\begin{proof}
Noting that along $L_0(h)$, we have $\frac{dy}{d\tau}=-x$. It follows from \eqref{1-I} that
\begin{align}\label{1-barI}
I_2(h)&=-\int_{\widehat{AB}}H^+_1 f_0(y)\frac{dy}{(-x)}+\int_{\widehat{BA}}H^+_1 f_0(y)\frac{dy}{(-x)}\nonumber\\
&=-\int_{\sqrt{2h}}^{-\sqrt{2h}}\frac{ H^+_1f_0(y)dy}{(-\sqrt{2h-y^2})}+\int_{-\sqrt{2h}}^{\sqrt{2h}} \frac{H^+_1 f_0(y)dy}{(\sqrt{2h-y^2})}=0,
\end{align}
where $H^+_1=\int_{0}^{y}g(y)dy.$

Since $f_0,~g_0$ are odd, we derive that
\begin{equation*}
\int_{\widehat{AB}}-\big(xf_0(y)-g_0(y)\big)dy=\int_{\sqrt{2h}}^{-\sqrt{2h}}-\big(\sqrt{2h-y^2}f_0(y)-g_0(y)\big)dy=0.
\end{equation*}
Thus from \eqref{1-I}, we have
\begin{equation}\label{1-hatI}
I_4(h)=0.
\end{equation}

Let us recall that $a(h,\lambda)$ and $b(h,\lambda)$ are two solutions of the equation
\begin{equation*}
H^+(0,y,\lambda)\equiv\frac{1}{2}y^2+\lambda \int_0^yg(y)dy=h.
\end{equation*}
Differentiating the above equation with respect to $\lambda$ from both sides, one has
\begin{equation*}
y\frac{\partial y}{\partial \lambda}+\int_0^yg(y)dy+\lambda g(y)\frac{\partial y}{\partial \lambda}=0,
\end{equation*}
which yields that
$$\frac{\partial y}{\partial \lambda}=\frac{-\int_0^yg(y)dy}{y+\lambda g(y)}.$$
 Then we have
\begin{align}\label{1-parab}
&\frac{\partial a}{\partial \lambda}\Big|_{\lambda=0}=\frac{\partial y}{\partial \lambda}\Big|_{ \lambda=0\atop y=\sqrt{2h} }=\frac{-\int_0^{\sqrt{2h}}g(y)dy}{\sqrt{2h}}, \nonumber\\
&\frac{\partial b}{\partial \lambda}\Big|_{\lambda=0}=\frac{\partial y}{\partial \lambda}\Big|_{\lambda=0\atop y=-\sqrt{2h}}=\frac{\int_0^{-\sqrt{2h}}g(y)dy}{\sqrt{2h}}.
\end{align}
Since $f_0,~g_0$ are odd, we can rewrite them as
\begin{equation}\label{1-M1-fg}
f_0(y)=\sum_{j=0}^{[\frac{m}{2}]} a_{2j+1}^{(0)} y^{2j+1},~~~g_0(y)=\sum_{j=0}^{[\frac{n}{2}]} b_{2j+1}^{(0)} y^{2j+1}.
\end{equation}

By \eqref{eqL}, \eqref{1-fg}, \eqref{1-Ii}, \eqref{1-parab} and \eqref{1-M1-fg}, one can get that
\begin{align}\label{1-tildeI}
I_3(h)&=2\Big[g_0(\sqrt{2h})\frac{-\int_0^{\sqrt{2h}}g(y)dy}{\sqrt{2h}}-g_0(-\sqrt{2h})\frac{\int_0^{-\sqrt{2h}}g(y)dy}{\sqrt{2h}}\Big]\nonumber\\
&=\frac{-2g_0(\sqrt{2h})}{\sqrt{2h}}\Big[\int_0^{\sqrt{2h}}\big(g(y)+g(-y)\big)dy\Big]\nonumber\\
&=2\sum_{j=0}^{[\frac{n}{2}]}- b_{2j+1}^{(0)} (\sqrt{2h})^{2j}\int_0^{\sqrt{2h}} 2 \sum_{j=0}^{[\frac{n}{2}]}c_{2j} y^{2j}dy \nonumber\\
&=\sum_{j=0}^{[\frac{n}{2}]} b^*_j h^j \cdot \sum_{j=0}^{[\frac{n}{2}]} c^*_j h^{j+\frac{1}{2}}=\sum_{l=0}^{2[\frac{n}{2}]}\Big( \sum_{ 0\leqslant i,j \leqslant [\frac{n}{2}]}^{i+j=l}  b_i^* c_j^* \Big)h^{l+\frac{1}{2}},
\end{align}
where
\begin{equation*}
b^*_i=-2^{i+1} b_{2i+1}^{(0)},~~~c^*_i= \frac{2^{i+\frac{3}{2}} c_{2i}}{2i+1},~~~i=0,\cdots,[\frac{n}{2}].
\end{equation*}

Hence, combining \eqref{1-m} with \eqref{1-final-Ii}, \eqref{1-barI}, \eqref{1-hatI} and \eqref{1-tildeI}, we have the following
\begin{equation}\label{1-m1}
M_1(h)= \sum\limits_{i=0}^{[\frac{m}{2}]}\tilde{a}^{(1)}_{i}h^{i+1}
+\sum_{l=0}^{2[\frac{n}{2}]} \tilde {b}_l h^{l+\frac{1}{2}},
\end{equation}
where
\begin{equation}\label{1-m1ab}
\tilde {b}_l=\left\{
\begin{aligned}
&\tilde{b}^{(1)}_l+ \sum_{ 0\leqslant i,j \leqslant [\frac{n}{2}]}^{i+j=l}  b_i^* c_j^*,&&l=0,\cdots,[\frac{n}{2}],\\
&\sum_{ 0\leqslant i,j \leqslant [\frac{n}{2}]}^{i+j=l}  b_i^* c_j^*,&&l=[\frac{n}{2}]+1,\cdots,2[\frac{n}{2}],
\end{aligned}\right.
\end{equation}
and $\tilde{a}_i^{(1)},~\tilde{b}_i^{(1)}$ are defined in \eqref{1-tildeab}.

We note that $\tilde{a}_i^{(1)},~\tilde{b}_i^{(1)}$, $b^*_i,~c^*_i$  are independent free parameters with respect to $a_{2i}^{(1)},~b_{2i}^{(1)}$, $b^{(0)}_{2i+1},~c_{2i}$ respectively. So it follows from \eqref{1-tildeab} and the above formula that $\tilde{a}_i^{(1)}$ and $\tilde{b}_i$ are also independent. Then by taking them as free parameters, Lemma \ref{lem1-M1} is proved immediately.
\end{proof}
From Lemma \ref{lem1-M1}, we know that the maximum number of zeros of $M(h,\lambda)$ is $[\frac{m}{2}]+2[\frac{n}{2}]+1$ for $0< \epsilon \ll \lambda \ll 1$ if $M_1(h)\not\equiv 0.$ Thus, we have proved the following theorem.
\begin{theorem}\label{thi-1}
Let $0< \epsilon \ll \lambda \ll 1$ and $M_0(h)\equiv0,~M_1(h)\not\equiv 0.$ Then \eqref{0} has at most $[\frac{m}{2}]+2[\frac{n}{2}]+1$ limit cycles bifurcating from the periodic orbits of the system $\dot{x}=y,~\dot{y}=-x.$ And this number can be reached.
\end{theorem}

Note that
\begin{align}\label{sgn}
&-\lambda \mbox{sgn}(y) g(x)-\epsilon \bigg[ y \Big(f_0(x)+\lambda f_1(x) \Big)+\mbox{sgn}(y)\Big(g_0(x)+\lambda g_1(x)\Big) \bigg] \nonumber\\
=&-\lambda\Big [y \bar{f}(x)+\mbox{sgn}(y)\bar{g}(x)\Big],
\end{align}
where
\begin{align*}
&\bar{f}(x)=\delta\Big(f_0(x)+\lambda f_1(x)\Big),\\
&\bar{g}(x)=g(x)+\delta\Big(g_0(x)+\lambda g_1(x)\Big),\\
&\delta=\frac{\epsilon}{\lambda}.
\end{align*}
Obviously, Theorem \ref{thi} follows from Theorem \ref{thi-1} directly. This ends the proof of Theorem \ref{thi}.

The following example is an illustration of Theorem \ref{thi}.

\begin{example}\label{ex1}
Take in system \eqref{0} $m=n=3$ and
\begin{align*}
&a_3^{(0)}=1,~a_0^{(1)}=-\frac{25}{\pi}, ~a_2^{(1)}=-\frac{10}{\pi}, \\
&b_3^{(0)}=-\frac{1}{4},~b_0^{(1)}=3\sqrt{2},~b_2^{(1)}=\frac{105\sqrt{2}}{16},\\
& c_2=\frac{3\sqrt{2}}{8},\\
&a^{(0)}_0=a^{(0)}_1=a^{(0)}_2=a^{(1)}_1=a^{(1)}_3=0,\\
&b^{(0)}_0=b^{(0)}_1=b^{(0)}_2=b^{(1)}_1=b^{(1)}_3=0,\\
&c_0=c_1=c_3=0.
\end{align*} 
Then by Lemmas \ref{lem1-M0} and \ref{lem1-M1}, we have
\begin{equation*}
M_1(h)=24h^{\frac{1}{2}}-50h+35h^{\frac{3}{2}}-10h^2+h^{\frac{5}{2}},
\end{equation*}
which has four nonzero simple roots $h=1,~4,~9,~16$. 
Therefore, it follows from \eqref{sgn} that the system below
\begin{equation}\label{ex1-1}
\left\{\begin{aligned}
&\dot{x}=y,\\
&\dot{y}=-x-\lambda\big[y\bar{f}+\mbox{sgn}(y)\bar{g}\big],
\end{aligned}\right.
\end{equation}
where 
\begin{align*}
&\bar{f}=\frac{\epsilon}{\lambda}x^3-\frac{10\epsilon}{\pi}x^2-\frac{25\epsilon}{\pi},\\
&\bar{g}=-\frac{\epsilon}{4\lambda}x^3+\Big(\frac{105\sqrt{2}}{16}\epsilon+\frac{3\sqrt{2}}{8}\Big)x^2+3\sqrt{2}\epsilon,
\end{align*}
has four limit cycles near $x^2+y^2=2,~8,~18,~32$ respectively. 
\end{example}

\begin{remark}
We can make a comparation of the number of limit cycles of \eqref{ex1-1} with a case of piecewise smooth system and a case of smooth system respectively.

First, consider 
\begin{equation}\label{ex1-2}
\left\{\begin{aligned}
&\dot{x}=y,\\
&\dot{y}=-x-\epsilon\big[y(a_3x^3+a_2x^2+a_1x+a_0)+\mbox{sgn}(y)(b_3x^3+b_2x^2+b_1x+b_0)\big].
\end{aligned}\right.
\end{equation}
From the formula of Melnikov function given by Theorem 1.1 in \citep{LH}, we have 
\begin{equation*}
M(h)=h^{\frac{1}{2}}\big(4\sqrt{2}b_0+2 a_0 \pi h^{\frac{1}{2}} +\frac{8\sqrt{2}}{3}b_2 h+a_2 \pi h^{\frac{3}{2}}\big),
\end{equation*}
which has at most three nonzero simple roots. It further indicates that if $M(h)\not\equiv 0$, \eqref{ex1-2} has at most three limit cycles, which is one less than \eqref{ex1-1}. 

Second, for the smooth system as follows
\begin{equation}\label{ex1-3}
\left\{\begin{aligned}
&\dot{x}=y,\\
&\dot{y}=-x-\epsilon \Big[y(a_3x^3+a_2x^2+a_1x+a_0) +k (b_3x^3+b_2x^2+b_1x+b_0)\Big],
\end{aligned}\right.
\end{equation}
with $k \neq 0,$  there exists at most one limit cycle by the first order Melnikov function $M(h)=(-2a_0\pi)h-(a_2\pi)h^2$.

Thus, these examples support the conclusion of Theorem \ref{thi}, which can be used to find more limit cycles.
\end{remark}

\section{Proof of Theorem \ref{thii}}
In this section, we study the following system
\begin{equation}\label{2}
\left \{
\begin{aligned}
&\dot{x}=y, \\
&\dot{y}=-x-\lambda\mbox{sgn}(x)g(x)-\epsilon \bigg[ y \Big(f_0(x)+\lambda f_1(x) \Big)+\mbox{sgn}(x)\Big(g_0(x)+\lambda g_1(x)\Big) \bigg],
\end{aligned}\right.
\end{equation}
where $g,~f_i,~g_i(i=0,1)$ are defined as in \eqref{1-fg}.
By Lemma \ref{lemx}, the first order Melnikov function $M(h,\lambda)$ of \eqref{2} has the following items
\begin{align}\label{2-M}
&M_0(h)=I_0(h),\nonumber\\
&M_1(h)=I_1(h)+I_2(h)+I_3(h),
\end{align}
where
\begin{align}\label{2-I}
&I_i(h)=\int_{\widehat{AB}}-\big(yf_i(x)+g_i(x)\big)dx+\int_{\widehat{BA}}-\big(yf_i(x)-g_i(x)\big)dx,~i=0,1, \nonumber\\
&I_2(h)=\int_{\widehat{AB}}H^+_1f_0(x)dt-\int_{\widehat{BA}}H^+_1f_0(x)dt,~\mbox{where}~ H^+_1=\int_{0}^{x}g(x)dx,\nonumber\\
&I_3(h)= \int_{\widehat{AB}}-\big(yf_0(x)-g_0(x)\big)dx,
\end{align}
and $A,~B,~L_0(h)$ are defined the same as before.

Then, we have a lemma below:
\begin{lemma}\label{2-M0}
The function $M_0(h)$ has at most $[\frac{m}{2}]$ zeros if it is not zero identically, and the number $[\frac{m}{2}]$ can be reached for some $f_0$. Further, if the polynomial $f_0$ is odd, then $M_0(h)\equiv0.$
\end{lemma}
\begin{proof}
For $i=0,1$, we have
\begin{equation*}
I_i(h)=-\oint\limits_{L_0(h)}yf_i(x)dx.
\end{equation*}
Taking \eqref{1-fg} into the above equation, by Green formula and using polar coordinate transformation $x=r\cos \theta,~y=r\sin \theta$, we have
\begin{align}\label{2-Ii}
I_i(h)&=-\iint\limits_{Int L_0} f_i(x)dxdy  =-\sum_{j=0}^{m} a_j^{(i)}  \iint\limits_{Int L_0}  x^j dxdy \nonumber\\
&= -\sum_{j=0}^{m} a_j^{(i)}  \int_0^{\sqrt{2h}}r^{j+1}dr \int_0^{2 \pi} (\cos\theta)^j d\theta\nonumber\\
&=\sum\limits_{j=0}^{[\frac{m}{2}]}\tilde{a}^{(i)}_{j}h^{j+1},~~i=0,1,
\end{align}
where
\begin{equation}\label{2-tildea}
\tilde{a}^{(i)}_{j}=\bigg(-\frac{2 \pi}{j+1} \prod\limits_{l=1}^{j}\frac{2l-1}{l}\bigg)a^{(i)}_{2j},~\tilde{a}^{(i)}_0=-2\pi a^{(i)}_0,~~~~j=1,2,\cdots,[\frac{m}{2}].
\end{equation}

For $i=0$, we have $M_0(h)=I_0(h)$. Taking $\tilde{a}^{(0)}_{j}$ as free parameters, we can conclude that $M_0(h)$ has at most $[\frac{m}{2}]$ isolated positive zeros for $h>0$ and this number can be reached. Furthermore, the second result of Lemma \ref{2-M0} is obviously obtained from \eqref{2-Ii} and \eqref{2-tildea}.
\end{proof}

Further, we have
\begin{lemma}\label{2-M1}
Suppose that the polynomial $f_0$ is odd with degree at most $m$. Then for $n\geq 1$ $($ $n=0$, resp.$)$, $M_1(h)$ has at most $2[\frac{m}{2}]+[\frac{n+1}{2}]$ $($ $2[\frac{m}{2}]+1$, resp.$)$ zeros and this number can be reached.
\end{lemma}
\begin{proof}
Let $A_1=(\sqrt{2h},0),~B_1=(-\sqrt{2h},0)$ be two points of $L_0$ intersected with $x$-axis, see Fig. \ref{fig2}.
Noting that $f_0$ is odd and $\frac{dx}{dt}=y$ along $L_0(h)$, then by \eqref{2-I}, one can derive that
\begin{align}\label{2-I2}
I_2(h)=&\int_{\widehat{AA_1}}H^+_1f_0(x)\frac{dx}{y}+\int_{\widehat{A_1B}}H^+_1f_0(x)\frac{dx}{y}\nonumber\\
&-\Big(\int_{\widehat{BB_1}}H^+_1f_0(x)\frac{dx}{y}+\int_{\widehat{B_1A}}H^+_1f_0(x)\frac{dx}{y}\Big) \nonumber\\
=&2\int_0^{\sqrt{2h}}\frac{H^+_1f_0(x)}{\sqrt{2h-x^2}}dx+2\int_0^{-\sqrt{2h}}\frac{H^+_1f_0(x)}{\sqrt{2h-x^2}}dx \nonumber\\
=&2\int_0^{\sqrt{2h}}\frac{f_0(x)}{\sqrt{2h-x^2}}\big(H^+_1(x)+H^+_1(-x)\big)dx,
\end{align}
where
\begin{align}\label{2-H+H}
H^+_1(x)+H^+_1(-x)&=\int_0^xg(x)dx+\int_0^{-x}g(x)dx=\int_0^x\big(g(x)-g(-x)\big)dx\nonumber\\
&=\int_0^x2\sum_{i=0}^{\tilde{n}}c_{2i+1}x^{2i+1}dx=2\sum_{i=0}^{\tilde{n}}\frac{c_{2i+1}}{2i+2}x^{2i+2},
\end{align}
where $\tilde{n}=[\frac{n+1}{2}]-1,$ for $n\geq1$.

Then taking \eqref{2-H+H} into \eqref{2-I2} and letting $f_0=\sum_{i=0}^{[\frac{m}{2}]} a_{2i+1}^{(0)}x^{2i+1}$, we have
\begin{align}\label{2-barI1}
I_2(h)&=4\int_0^{\sqrt{2h}} \frac{1}{\sqrt{2h-x^2}} \sum_{i=0}^{[\frac{m}{2}]} a_{2i+1}^{(0)}x^{2i+1} \cdot \sum_{i=0}^{\tilde{n}}\frac{c_{2i+1}}{2i+2}x^{2i+2} dx \nonumber\\
&=\int_0^{\sqrt{2h}} 4\sum_{l=0}^{[\frac{m}{2}]+\tilde{n}}\Big(\sum_{0\leqslant i \leqslant [\frac{m}{2}] \atop 0\leqslant j \leqslant \tilde{n} }^{i+j=l}  \frac{a_{2i+1}^{(0)}c_{2j+1}}{2j+2}\Big) \frac{x^{2l+3}}{\sqrt{2h-x^2}} dx\nonumber\\
&=\sum_{l=0}^{[\frac{m}{2}]+\tilde{n}} \Big(\sum_{0\leqslant i \leqslant [\frac{m}{2}] \atop 0\leqslant j \leqslant \tilde{n} }^{i+j=l}  \frac{2a_{2i+1}^{(0)}c_{2j+1}}{j+1} \Big)J_l,
\end{align}
where for each $l$,
\begin{equation*}
J_l=\int_0^{\sqrt{2h}} \frac{x^{2l+3}}{\sqrt{2h-x^2}} dx.
\end{equation*}
Letting $t=\sqrt{2h-x^2}$, then for each $J_l$, we have
\begin{align*}
J_l&=\int_0^{\sqrt{2h}} (2h-t^2)^{l+1}dt \nonumber\\
&=\sum_{k=0}^{l+1} C^k_{l+1} (-1)^{k} (2h)^{l+1-k} \int_0^{\sqrt{2h}} t^{2k} dt  \nonumber\\
&= \Big(\sum_{k=0}^{l+1} \frac{C^k_{l+1} (-1)^{k}2^{l+\frac{3}{2}}}{2k+1}\Big)h^{l+\frac{3}{2}}.
\end{align*}
By taking the above equation in \eqref{2-barI1}, we have
\begin{equation}\label{2-barI}
I_2(h)=\sum_{l=0}^{[\frac{m}{2}]+\tilde{n}} a_l^* h ^{l+\frac{3}{2}},
\end{equation}
where
\begin{align}\label{2-a*}
&a_l^*=\Big(\sum_{k=0}^{l+1} \frac{C^k_{l+1} (-1)^{k}2^{l+\frac{3}{2}}}{2k+1}\Big)\cdot\Big(\sum_{0\leqslant i \leqslant [\frac{m}{2}] \atop 0\leqslant j \leqslant \tilde{n} }^{i+j=l}  \frac{2a_{2i+1}^{(0)}c_{2j+1}}{j+1} \Big),\nonumber\\
&\mbox{for}~~l=0,\cdots,[\frac{m}{2}]+\tilde{n},~~\tilde{n}=[\frac{n+1}{2}]-1,~~n\geq1.
\end{align}

Then, denoting the set $\widehat{AB}\bigcup \overrightarrow{BA}$ by $L_0^+(h)$, it follows from \eqref{2-I} that
\begin{align*}
I_3(h)&=-\int_{\widehat{AB}}yf_0(x)dx=-\Big(\int_{\widehat{AB}}yf_0(x)dx+\int_{\overrightarrow{BA}}yf_0(x)dx-\int_{\overrightarrow{BA}}yf_0(x)dx\Big)\nonumber\\
&=-\oint\limits_{L_0^+} yf_0(x)dx=-\iint\limits_{IntL_0^+}f_0(x)dxdy=-\sum_{i=0}^{[\frac{m}{2}]}a_{2i+1}^{(0)}\iint\limits_{IntL_0^+}x^{2i+1}dxdy.
\end{align*}
Again by using polar coordinate transformation $x=r\cos\theta,~y=r\sin\theta$, we can obtain
\begin{align}\label{2-hatI}
I_3(h)&=-\sum_{i=0}^{[\frac{m}{2}]}a_{2i+1}^{(0)} \int_0^{\sqrt{2h}}r^{2i+2}dr\cdot\int_{-\frac{\pi}{2}}^{\frac{\pi}{2}}(\cos\theta)^{2i+1}d\theta \nonumber\\
&=-\sum_{i=0}^{[\frac{m}{2}]}a_{2i+1}^{(0)} \frac{(\sqrt{2h})^{2i+3}}{2i+3}2\int_0^{\frac{\pi}{2}}(\cos\theta)^{2i+1}d\theta\nonumber\\
&=\sum_{i=0}^{[\frac{m}{2}]} \hat{a}_i h^{i+\frac{3}{2}},
\end{align}
where
\begin{equation}\label{2-hata}
\hat{a}_i=\Big(-\frac{2^{i+\frac{5}{2}}}{2i+3} \int_0^{\frac{\pi}{2}}(\cos\theta)^{2i+1}d\theta\Big)a_{2i+1}^{(0)},~~i=0,\cdots,[\frac{m}{2}].
\end{equation}
Hence, for $n\geq1$, substituting \eqref{2-Ii}, \eqref{2-barI} and \eqref{2-hatI} into \eqref{2-M}, we can obtain that
\begin{equation} \label{2-mn1}
M_1(h)=\sum\limits_{l=0}^{[\frac{m}{2}]}\tilde{a}^{(1)}_{l}h^{l+1}+\sum_{l=0}^{[\frac{m}{2}]+[\frac{n+1}{2}]-1} \tilde{b}_l h ^{l+\frac{3}{2}},
\end{equation}
where
\begin{equation*}
\tilde {b}_l=\left\{
\begin{aligned}
& a_l^*+ \hat{a}_l,&&l=0,\cdots,[\frac{m}{2}],\\
&a_l^*,&&l=[\frac{m}{2}]+1,\cdots,[\frac{m}{2}]+[\frac{n+1}{2}]-1,
\end{aligned}\right.
\end{equation*}
where $\tilde{a}^{(1)}_l,~a_l^*,~\hat{a}_l$ are defined as in \eqref{2-tildea}, \eqref{2-a*} and \eqref{2-hata} respectively.

For $n=0$, from \eqref{2-I2} and \eqref{2-H+H}, we have $I_2(h)=0.$ Thus, together with \eqref{2-Ii} and \eqref{2-hatI}, it implies that
\begin{equation}\label{2-mn0}
M_1(h)=\sum\limits_{l=0}^{[\frac{m}{2}]}\tilde{a}^{(1)}_{j}h^{j+1}+\sum_{l=0}^{[\frac{m}{2}]} \hat{a}_j h ^{j+\frac{3}{2}},
\end{equation}
where $\tilde{a}_j^{(1)},~\hat{a}_j$ are defined as in \eqref{2-tildea} and \eqref{2-hata} respectively.

In summary, by taking $\tilde{a}^{(1)}_l,~\tilde {b}_l$ ({\em $\tilde{a}_j^{(1)},~\hat{a}_j$ resp.})  as free parameters, we can see that Eq. \eqref{2-mn1} ({ \em Eq. \eqref{2-mn0} resp.}) has at most $2[\frac{m}{2}]+[\frac{n+1}{2}]$ ({\em $2[\frac{m}{2}]+1$ resp.}) zeros if $M_1(h)$ is not zero identically.
\end{proof}

Thus we have proved the following theorem:
\begin{theorem}\label{thii-1}
Let $0< \epsilon \ll \lambda \ll 1$ and $M_0(h)\equiv0,~M_1(h)\not\equiv 0.$ Then for $n\geq1$ $($$n=0$, resp.$)$, system \eqref{2} has at most $2[\frac{m}{2}]+[\frac{n+1}{2}]$ $($$2[\frac{m}{2}]+1$, resp.$)$ limit cycles bifurcating from the periodic orbits of the system $\dot{x}=y,~\dot{y}=-x.$ And this number can be reached.
\end{theorem}

Then if we replace $\mbox{sgn}(y)$ with $\mbox{sgn}(x)$ into \eqref{sgn}, Theorem \ref{thii} follows from \eqref{sgn} and Theorem \ref{thii-1} directly.

Next we give an example to illustrate Theorem \ref{thii}:
\begin{example}
Take in \eqref{2} $m=n=3$ and
\begin{align*}
&a_1^{(0)}=357\sqrt{2}, ~a_3^{(0)}=-\frac{93653}{260}\sqrt{2},~a_0^{(1)}=-\frac{120}{\pi}, ~a_2^{(1)}=-\frac{300}{\pi}, \\
&c_0=0,~ c_1=\frac{1}{4},~c_3=-\frac{4133}{7616},\\
& a_0^{(0)}=a_2^{(0)}=a_1^{(1)}=a_3^{(1)}=0, ~b_j^{(i)}=0~(i=0,1,~j=0,\cdots,3).
\end{align*} 
Similar to Example \ref{ex1}, it comes from the proof of Theorem \ref{thii} that the following system 
\begin{equation}\label{ex2}
\left\{\begin{aligned}
&\dot{x}=y,\nonumber\\
&\dot{y}=-x-\lambda\big[y\hat{f}+\mbox{sgn}(x)\hat{g}\big],
\end{aligned}\right.
\end{equation}
where 
\begin{align*}
&\hat{f}=\frac{93653\sqrt{2}\epsilon}{260\lambda}x^3-\frac{300\epsilon}{\pi}x^2+\frac{357\sqrt{2}\epsilon}{\lambda}x-\frac{120\epsilon}{\pi},\\
&\hat{g}=-\frac{4133}{7616}x^3+\frac{1}{4}x
\end{align*}
has four limit cyles.
\end{example}

\section*{Acknowledgment}
The author would like to give thanks to professor Maoan Han for introducing the topic together with his helpful discussions during the preparation of the paper and the anonymous referees for the valuable suggestions.


\begin{thebibliography}{12}
\newcommand{\enquote}[1]{``#1''}
\providecommand{\natexlab}[1]{#1}
\providecommand{\url}[1]{\texttt{#1}}
\providecommand{\urlprefix}{URL }
\expandafter\ifx\csname urlstyle\endcsname\relax
  \providecommand{\doi}[1]{doi:\discretionary{}{}{}#1}\else
  \providecommand{\doi}{doi:\discretionary{}{}{}\begingroup
  \urlstyle{rm}\Url}\fi

\bibitem[{Buica \& Llibre(2004)}]{bj2004}
Buica, A. \& Llibre, J. [2004] \enquote{Averaging methods for finding periodic
  orbits via Brouwer degree,} \emph{Bull. Sci. math.} \textbf{128},  7--22.

\bibitem[{Han \& Romanovski(2013)}]{HR}
Han, M. \& Romanovski, V.~G. [2013] \enquote{On the number of limit cycles of
  polynomial Li\'enard systems,} \emph{Nonlinear Anal-Real.} \textbf{14},
  1655--1668.

\bibitem[{Han \& Sheng(2015)}]{MS}
Han, M. \& Sheng, L. [2015] \enquote{Bifurcation of limit cycles in piecewise
  smooth systems via Melnikov function,} \emph{J. Appl. Anal. Comp.}
  \textbf{5},  809--815.

\bibitem[{Han \& Xiong(2014)}]{HX}
Han, M. \& Xiong, Y. [2014] \enquote{Limit cycle bifurcations in a class of
  near--Hamiltonian systems with multiple parameters,} \emph{Chaos Solitons
  Fractals} \textbf{68},  20--29.

\bibitem[{Li \& Llibre(2012)}]{LL}
Li, C. \& Llibre, J. [2012] \enquote{Uniqueness of limit cycles for Li\'enard
  differential equations of degree four,} \emph{J. Diff. Eqs.} \textbf{252},
  3142--3162.

\bibitem[{Liang \& Han(2012)}]{LiangH}
Liang, F. \& Han, M. [2012] \enquote{Limit cycles near generalized homoclinic
  and double homoclinic loops in piecewise smooth systems,} \emph{Chaos
  Solitons Fractals} \textbf{45},  454--464.

\bibitem[{Li\'enard(1928)}]{Lie}
Li\'enard, A. [1928] \enquote{Etude des oscillations entrenues,} \emph{Rev.
  Gen. Electr.} \textbf{23},  901--912.

\bibitem[{Liu \& Han(2010)}]{LH}
Liu, X. \& Han, M. [2010] \enquote{Bifurcation of limit cycles by perturbing
  piecewise Hamiltonian systems,} \emph{Int. J. Bifurcation and Chaos}
  \textbf{20},  1379--1390.

\bibitem[{Martins \& Mereu(2014)}]{MM}
Martins, R.~M. \& Mereu, A.~C. [2014] \enquote{Limit cycles in discontinuous
  classical Li\'enard equations,} \emph{Nonlinear Anal-Real.} \textbf{20},
  67--73.

\bibitem[{Xiong(2015)}]{X}
Xiong, Y. [2015] \enquote{Limit cycle bifurcations by perturbing piecewise
  smooth Hamiltonian systems with multiple parameters,} \emph{J. Math. Anal.
  Appl.} \textbf{421},  260--275.

\bibitem[{Xiong \& Han(2014)}]{X2}
Xiong, Y. \& Han, M. [2014] \enquote{New lower bounds for the Hilbert number of
  polynomial systems of Li\'enard type,} \emph{J. Differential Equation}
  \textbf{257},  2565--2590.

\bibitem[{Yang(2012)}]{junmin}
Yang, J. [2012] \enquote{On the limit cycles of a kind of Li\'enard system with
  a nilpotent center under perturbations,} \emph{J. Appl. Anal. Comp.}
  \textbf{2},  325--339.

\end{thebibliography}
\end{document}